\documentclass[12pt,english]{article}
\pdfoutput=1
\usepackage[T1]{fontenc}
\usepackage[latin9]{inputenc}
\usepackage{geometry}
\usepackage{array}
\usepackage{calc}
\usepackage{mathtools}
\usepackage{amsmath}
\usepackage{amsthm}
\usepackage{amssymb}

\makeatletter

\providecommand{\tabularnewline}{\\}

\numberwithin{equation}{section}
\theoremstyle{plain}
\newtheorem{thm}{\protect\theoremname}[section]
\theoremstyle{plain}
\newtheorem{conjecture}[thm]{\protect\conjecturename}
\theoremstyle{plain}
\newtheorem{lem}[thm]{\protect\lemmaname}
\theoremstyle{plain}
\newtheorem{prop}[thm]{\protect\propositionname}
\theoremstyle{definition}
\newtheorem{example}[thm]{\protect\examplename}
\theoremstyle{definition}
\newtheorem{problem}[thm]{\protect\problemname}

\usepackage{bm}
\usepackage{tikz}
\usetikzlibrary{arrows}
\usepackage{hyperref}
\usepackage{needspace}
\usepackage{colortbl}
\usepackage{xparse}
\usepackage{caption}
\usepackage{microtype}
\usepackage{titlesec}
\titleformat{\section}
  {\normalfont\large\bfseries}{\thesection.}{.8ex}{} 
\titleformat{\subsection}
  {\normalfont\bfseries}{\thesubsection.}{.8ex}{} 
\usetikzlibrary{shapes}
\tikzstyle{pathdefault}=[draw, line width=1, solid, color=black]
\tikzstyle{nodedefault}=[circle, inner sep=1.1, fill=black]
\tikzstyle{nodered}=[circle, inner sep=1.1, fill=red]
\tikzstyle{nodeblue}=[circle, inner sep=1.1, fill=blue]
\tikzstyle{empty}=[]
\tikzstyle{nodeellipsis}=[circle, inner sep=0.5, fill=black]
\tikzstyle{pathcolor1}=[draw, line width=1, solid, color=red]
\tikzstyle{pathcolor2}=[draw, line width=1, solid, color=blue]
\tikzstyle{pathcolorlight}=[draw, line width=1, dotted, color=blue]
\tikzstyle{arbpathcolor0}=[line width=1, dashdotted, color=black]
\tikzstyle{arbpathcolor1}=[line width=1, densely dashed, color=red]
\tikzstyle{arbpathdefault}=[line width=1, densely dotted, color=blue]
\newcounter{id}
\newcommand{\drawlinedotswithstyle}[4]{
 \def\x{{#3}}
 \def\y{{#4}}
 \tikzstyle{thispathstyle}=[#1]
 \tikzstyle{thisnodestyle}=[#2]
 \setcounter{id}{-1} 
 \foreach \j in {#3}{\stepcounter{id}} 
 \foreach \i in {1,...,\the\value{id}}{  
  \path[thispathstyle] (\x[\i],\y[\i]) --(\x[\i-1],\y[\i-1]); 
 }
 \foreach \i in {1,...,\the\value{id}}{  
  \node[thisnodestyle] at (\x[\i],\y[\i]) {}; 
 }
 \node[thisnodestyle] at (\x[0],\y[0]) {}; 
}

\DeclareDocumentCommand{\drawlinedots}{ O{pathdefault} O{nodedefault} m m}{\drawlinedotswithstyle{#1}{#2}{#3}{#4}}
\DeclareDocumentCommand{\drawlinedotsred}{ O{pathcolor1} O{nodered} m m}{\drawlinedotswithstyle{#1}{#2}{#3}{#4}}
\DeclareDocumentCommand{\drawlinedotsblue}{ O{pathcolor2} O{nodeblue} m m}{\drawlinedotswithstyle{#1}{#2}{#3}{#4}}
\usetikzlibrary{arrows.meta}
\tikzset{
  mapsto/.style={
    very thick,
    -{Latex[length=3mm,width=2mm]}
  }
}
\let\originalleft\left
\let\originalright\right
\renewcommand{\left}{\mathopen{}\mathclose\bgroup\originalleft}
\renewcommand{\right}{\aftergroup\egroup\originalright}
\newcommand{\leqnomode}{\tagsleft@true\let\veqno\@@leqno}
\newcommand{\reqnomode}{\tagsleft@false\let\veqno\@@eqno}
\reqnomode
\makeatother

\usepackage{babel}
\providecommand{\conjecturename}{Conjecture}
\providecommand{\examplename}{Example}
\providecommand{\lemmaname}{Lemma}
\providecommand{\problemname}{Problem}
\providecommand{\propositionname}{Proposition}
\providecommand{\theoremname}{Theorem}

\begin{document}
\global\long\def\des{\operatorname{des}}%

\global\long\def\ides{\operatorname{ides}}%

\global\long\def\pix{\operatorname{pix}}%

\global\long\def\fix{\operatorname{fix}}%

\global\long\def\st{\operatorname{st}}%

\global\long\def\asc{\operatorname{asc}}%

\global\long\def\std{\operatorname{std}}%

\global\long\def\Comp{\operatorname{Comp}}%

\global\long\def\ad{\operatorname{ad}}%

\title{Bijections between pattern-avoiding \\
derangements and desarrangements}
\author{Alyssa G.\ Henke, Derek H.\ Stephens, and Yan Zhuang\\
Department of Mathematics and Computer Science\\
Davidson College\texttt{}~\\
\texttt{\{alhenke, destephens, yazhuang\}@davidson.edu}}
\maketitle
\begin{abstract}
Derangements are permutations without fixed points, and are in bijection
with desarrangements: permutations whose first non-descent is even,
or equivalently, permutations without ``pixed points''. Bsila, Cox,
Hugo, Styron, and Zhuang recently proved a theorem characterizing
all $\Pi\subseteq\mathfrak{S}_{3}$, such that $1\leq\left|\Pi\right|\leq3$,
for which the number of derangements avoiding all patterns in $\Pi$
is equal to the number of desarrangements avoiding all patterns in
$\Pi$. They left finding a bijective proof of this theorem as an
open problem, and posed a related conjecture concerning the distributions
of fixed points and pixed points over pattern avoidance classes. In
this paper, we give bijective proofs of this theorem and conjecture.{\let\thefootnote\relax}
\end{abstract}
\textbf{\small{}Keywords: }{\small{}derangements, desarrangements,
pattern avoidance, fixed points, pixed points, descents }{\small\par}

\section{\label{s-intro}Introduction}

\subsection{Definitions and background}

Let $\mathfrak{S}_{n}$ denote the symmetric group of permutations
of $[n]\coloneqq\{1,2,\dots,n\}$. By convention, $\mathfrak{S}_{0}$
consists of the empty permutation. The \textit{length} of a permutation
$\pi$ refers to its number of letters (so that $\pi\in\mathfrak{S}_{n}$
has length $n$). We sometimes use the term ``permutation'' for sequences
of distinct positive integers, as opposed to specifically permutations
in $\mathfrak{S}_{n}$, although some of the following definitions
only apply to the latter.

We say that $i\in[n]$ is a \textit{fixed point} of $\pi=\pi_{1}\pi_{2}\cdots\pi_{n}\in\mathfrak{S}_{n}$
if $\pi_{i}=i$, and we let $\fix(\pi)$ denote the number of fixed
points of $\pi$. Let ${\cal D}_{n}$ be the set of \textit{derangements}\textemdash permutations
with no fixed points\textemdash in $\mathfrak{S}_{n}$. For example,
we have $31254\in\mathcal{D}_{5}$, but $13542\notin\mathcal{D}_{5}$
as it has two fixed points, namely $1$ and $4$. 

Let $d_{n}\coloneqq\left|\mathcal{D}_{n}\right|$, the $n$th \textit{derangement
number}. The study of these numbers is classical, dating back to the
work of de Montmort \cite{Montmort1980} in the early 18th century.
The first several derangement numbers are displayed in Table \ref{tb-dernum};
see also entry A000166 of the On-Line Encyclopedia of Integer Sequences
(OEIS) \cite{oeis}.\renewcommand{\arraystretch}{1.2}
\begin{table}
\begin{centering}
\begin{tabular}{c|c|c|c|c|c|c|c|c|c|c|c|c}
$n$ & $0$ & $1$ & $2$ & $3$ & $4$ & $5$ & $6$ & $7$ & $8$ & $9$ & $10$ & $11$\tabularnewline
\hline 
$d_{n}$ & $1$ & $0$ & $1$ & $2$ & $9$ & $44$ & $265$ & $1854$ & $14833$ & $133496$ & $1334961$ & $14684570$\tabularnewline
\end{tabular}
\par\end{centering}
\caption{\label{tb-dernum}The first several derangement numbers.}
\bigskip{}
\end{table}

Another combinatorial model for the derangement numbers is given by
desarrangements. For a permutation $\pi=\pi_{1}\pi_{2}\cdots\pi_{n}$,
we say that $i\in[n-1]$ is an \textit{ascent} of $\pi$ if $\pi_{i}<\pi_{i+1}$,
and is a \textit{descent} of $\pi$ if $\pi_{i}>\pi_{i+1}$. Observe
that an index $i\in[n]$ is not a descent of $\pi$ precisely when
either $i$ is an ascent of $\pi$ or $i=n$. Let $\asc(\pi)$ be
the number of ascents and $\des(\pi)$ the number of descents of $\pi$.

A \textit{desarrangement} is a permutation whose first non-descent
is even. (By convention, the empty permutation is also considered
to be a desarrangement.) Let $\mathfrak{D}_{n}$ denote the set of
desarrangements in $\mathfrak{S}_{n}$. For instance, $85216437\in\mathfrak{D}_{8}$
because its first ascent, $4$, is even. Note that the decreasing
permutation $n\cdots21$ is a desarrangement if and only if $n$ is
even.

Desarrangements were introduced by D\'{e}sarm\'{e}nien \cite{Desarmenien1984},
who proved that they are in bijection with derangements and used them
to give a combinatorial proof of the recurrence $d_{n}=nd_{n-1}+(-1)^{n}$
for the derangement numbers. For our purposes, it is instructive to
review D\'{e}sarm\'{e}nien's bijection $\Phi\colon\mathcal{D}_{n}\rightarrow\mathfrak{D}_{n}$,
which we also illustrate with an example:
\begin{center}
\vspace{-10bp}
\begin{tabular}{l>{\raggedright}p{4.25in}>{\raggedright}p{0.03in}l}
1. & Take a derangement $\pi$. &  & $\pi=831967524$\tabularnewline
2. & Write $\pi$ in cycle notation, such that each cycle contains its
smallest entry in the second position, and where the cycles are ordered
in decreasing order of their smallest entries. &  & $\pi=(756)(94)(3182)$\tabularnewline
3. & Remove the parentheses; the result is a desarrangement. &  & $\Phi(\pi)=756943182$\tabularnewline
\end{tabular}
\par\end{center}

A prominent theme in permutation enumeration is the study of pattern
avoidance. For a permutation $\tau$ of length $n$\textemdash not
necessarily in $\mathfrak{S}_{n}$\textemdash its \textit{standardization}
$\std(\tau)$ is defined to be the permutation in $\mathfrak{S}_{n}$
obtained by replacing the smallest letter of $\tau$ by 1, the second
smallest by 2, and so on. For example, we have $\std(635)=312$. Given
permutations $\sigma$ and $\pi$, an \textit{occurrence} of $\sigma$
in $\pi$ is a subsequence of $\pi$ whose standardization is $\sigma$,
and $\pi$ is said to \textit{avoid} $\sigma$ (as a \textit{pattern})\textemdash and
that $\pi$ is \textit{$\sigma$-avoiding}\textemdash if $\pi$ contains
no occurrences of $\sigma$. For instance, $635$ is an occurrence
of $\sigma=312$ in $\pi=163254$, but $\pi=163254$ avoids $\sigma=231$. 

Let $\mathfrak{S}_{n}(\sigma)$ denote the set of permutations in
$\mathfrak{S}_{n}$ which avoid the pattern $\sigma$. So, as we observed,
$163254\in\mathfrak{S}_{6}(231$). It is well known that, for any
length 3 pattern $\sigma\in\mathfrak{S}_{3}$, the number of permutations
in $\mathfrak{S}_{n}(\sigma)$ is equal to the $n$th Catalan number.

If we have a set $\Pi$ of patterns, let $\mathfrak{S}_{n}(\Pi)$
denote the set of permutations in $\mathfrak{S}_{n}$ which avoid
every pattern in $\Pi$; such permutations are called $\Pi$-\textit{avoiding}.\footnote{For a specific $\Pi$, we will often omit the curly braces enclosing
the elements of $\Pi$ when writing out $\mathfrak{S}_{n}(\Pi)$ and
in similar notations.} In their seminal work on permutation pattern avoidance, Simion and
Schmidt \cite{Simion1985} achieved the enumeration of the avoidance
classes $\mathfrak{S}_{n}(\Pi)$ for all length 3 pattern sets $\Pi\subseteq\mathfrak{S}_{3}$. 

Because this paper concerns derangements and desarrangements that
avoid prescribed patterns, let us introduce the notation $\mathcal{D}_{n}(\Pi)$
for the set of derangements in $\mathfrak{S}_{n}(\Pi)$, and $\mathfrak{D}_{n}(\Pi)$
for the set of desarrangements in $\mathfrak{S}_{n}(\Pi)$.

A number of works in the permutation patterns literature study distributions
of permutation statistics over avoidance classes. Most relevant for
us are the works of Robertson, Saracino, and Zeilberger \cite{Robertson2002}
and of Mansour and Robertson \cite{Mansour2002}, who determined the
distribution of the fixed points statistic $\fix$ over all $\mathfrak{S}_{n}(\Pi)$
where $\Pi\subseteq\mathfrak{S}_{3}$ and $1\leq\left|\Pi\right|\leq3$.
These results specialize to enumerations for sets $\mathcal{D}_{n}(\Pi)$
of pattern-avoiding derangements.

More recently, Bsila, Cox, Hugo, Styron, and Zhuang \cite{Bsila2025}
studied the desarrangement avoidance classes $\mathfrak{D}_{n}(\Pi)$
for all $\Pi\subseteq\mathfrak{S}_{3}$. By comparing their results
with those of Robertson\textendash Saracino\textendash Zeilberger
and Mansour\textendash Robertson, Bsila et al.~obtained the following
theorem.
\begin{thm}[{\cite[Theorem 3.28]{Bsila2025}}]
\label{t-derdes}Let $\Pi\subseteq\mathfrak{S}_{3}$ with $1\leq\left|\Pi\right|\leq3$.
Then $\left|\mathcal{D}_{n}(\Pi)\right|=\left|\mathfrak{D}_{n}(\Pi)\right|$
for all $n\geq0$ if and only if $\Pi$ is one of the following: $\{132\}$,
$\{132,312\}$, $\{132,321\}$, $\{213,231\}$, $\{123,132,312\}$,
$\{123,213,231\}$, $\{123,312,321\}$, $\{132,312,321\}$, $\{213,231,312\}$,
or $\{213,231,321\}$.
\end{thm}

Bsila et al.~posed the open problem of proving Theorem \ref{t-derdes}
bijectively. One may naively hope that D\'{e}sarm\'{e}nien's bijection
$\Phi\colon\mathcal{D}_{n}\rightarrow\mathfrak{D}_{n}$ restricts
to bijections from $\mathcal{D}_{n}(\Pi)$ to $\mathfrak{D}_{n}(\Pi)$
for some of these $\Pi$, but taking a permutation in cycle notation
and removing the parentheses does not preserve pattern avoidance restrictions.
Instead, novel bijections are needed.

Just as derangements are permutations without fixed points, desarrangements
can be defined as permutations without ``pixed points''. As shown
by Foata and Han \cite{Foata2008a}, every permutation $\pi$ can
be written uniquely as a concatenation $\pi=\iota\delta$ where $\iota$
is an increasing permutation and $\delta$ is a desarrangement. For
example, $\pi=13675842$ decomposes as $\iota=136$ and $\delta=75842$,
and $\delta$ is indeed a desarrangement because its first ascent
$2$ is even. Foata and Han call this the \textit{pixed factorization}
of $\pi$ and the letters of $\iota$ \textit{pixed points}. Let $\pix(\pi)$
denote the number of pixed points of a permutation $\pi$. For example,
$\pix(13675842)=3$.

The bijection $\Phi\colon\mathcal{D}_{n}\rightarrow\mathfrak{D}_{n}$
can be extended to a bijection from $\mathfrak{S}_{n}$ to $\mathfrak{S}_{n}$\textemdash which
we also denote $\Phi$ by a slight abuse of notation\textemdash satisfying
$\fix(\pi)=\pix(\Phi(\pi))$ for all $\pi\in\mathfrak{S}_{n}$. We
simply remove the fixed points from $\pi$ in cycle notation, apply
$\Phi$ to the resulting derangement, and then prepend the fixed points
in ascending order (where they become pixed points). Thus, the statistics
$\fix$ and $\pix$ are equidistributed over $\mathfrak{S}_{n}$.

It is natural to wonder whether $\fix$ and $\pix$ are equidistributed
over any avoidance classes $\mathfrak{S}_{n}(\Pi)$. Bsila et al.~offered
the conjecture below which addresses this question. In short, for
all pattern sets listed in Theorem \ref{t-derdes} except for $\Pi=\{132\}$,
the equinumerosity $\left|\mathcal{D}_{n}(\Pi)\right|=\left|\mathfrak{D}_{n}(\Pi)\right|$
can be extended to an equidistribution between $\fix$ and $\pix$
over $\mathfrak{S}_{n}(\Pi)$.
\begin{conjecture}[{\cite[Conjecture 3.29]{Bsila2025}}]
\label{cj-main}Let $\Pi\subseteq\mathfrak{S}_{3}$ with $1\leq\left|\Pi\right|\leq3$.
Then $\fix$ and $\pix$ are equidistributed over $\mathfrak{S}_{n}(\Pi)$
for all $n\geq0$ if and only if $\Pi$ is one of the following: $\{132,312\}$,
$\{132,321\}$, $\{213,231\}$, $\{123,132,312\}$, $\{123,213,231\}$,
$\{123,312,321\}$, $\{132,312,321\}$, $\{213,231,312\}$, or $\{213,231,321\}$.
\end{conjecture}

\subsection{Overview of results}

In this paper, we provide bijective proofs for Theorem \ref{t-derdes}
and Conjecture \ref{cj-main}. All of our bijections for Conjecture
\ref{cj-main} preserve the descent number, so in fact we show that
$\fix$ and $\des$ are jointly equidistributed with $\pix$ and $\des$
over the listed avoidance classes. Our main result is the following.
\begin{thm}
\label{t-main}Let $\Pi$ be $\{132,312\}$, $\{132,321\}$, $\{213,231\}$,
$\{123,132,312\}$, $\{123,213,231\}$, $\{123,312,321\}$, $\{132,312,321\}$,
$\{213,231,312\}$, or $\{213,231,321\}$. Then, for all $n\geq0$,
the statistics $(\fix,\des)$ and $(\pix,\des)$ are equidistributed
over $\mathfrak{S}_{n}(\Pi)$.
\end{thm}

One can check that $\fix$ and $\pix$ have different distributions
over $\mathfrak{S}_{n}(\Pi)$ for all other $\Pi\subseteq\mathfrak{S}_{3}$
with $1\leq\left|\Pi\right|\leq3$. Hence, Conjecture \ref{cj-main}
follows.\footnote{We note in advance that most of our bijections used to prove Theorem
\ref{t-main} have definitions that are only valid for $n\geq2$,
but the equidistributions are trivial for $n=0$ and $n=1$.}

All of our bijections giving the equidistributions in Theorem \ref{t-main}
restrict to bijections between the corresponding $\mathcal{D}_{n}(\Pi)$
and $\mathfrak{D}_{n}(\Pi)$, so our bijective resolution to Conjecture
\ref{cj-main} automatically yields a bijective proof of Theorem \ref{t-derdes}
for every pattern set except $\Pi=\{132\}$. We will separately construct
a bijection from $\mathcal{D}_{n}(132)$ to $\mathfrak{D}_{n}(132)$.

While it is more common to encode the descent class of a permutation
using the descent set or the descent composition, it will be more
convenient for our purposes to use a word. Let $\mathcal{W}_{n}$
be the set of words of length $n$ over the alphabet $\{A,D\}$, and
define the \textit{$AD$-word} of a permutation $\pi=\pi_{1}\pi_{2}\cdots\pi_{n}$
to be the word 
\[
\ad(\pi)\coloneqq w_{1}w_{2}\cdots w_{n-1}\in\mathcal{W}_{n-1}\quad\text{where}\quad w_{i}=\begin{cases}
A, & \text{if }\pi_{i}<\pi_{i+1},\\
D, & \text{if }\pi_{i}>\pi_{i+1},
\end{cases}
\]
for each $i\in[n-1]$. For example, $\ad(153462)=ADAAD$. In almost
all of the $\mathfrak{S}_{n}(\Pi)$ covered in Theorem \ref{t-main},
permutations are uniquely determined by their descent class, so most
of our bijections will essentially act on $AD$-words.

Our paper is organized as follows:
\begin{itemize}
\item In Section \ref{s-213-231}, we construct a bijection $\Psi\colon\mathcal{W}_{n}\rightarrow\mathcal{W}_{n}$
which sends $AD$-words of $\{213,231\}$-avoiding permutations with
a given number of fixed points to $AD$-words of $\{213,231\}$-avoiding
permutations with that number of pixed points. This directly leads
to a bijection $\Phi_{213,231}\colon\mathfrak{S}_{n}(213,231)\rightarrow\mathfrak{S}_{n}(213,231)$
which transforms the number of fixed points into the number of pixed
points while preserving the number of descents, hence proving Theorem
\ref{t-main} for $\Pi=\{213,231\}$.
\item In Section \ref{s-others}, we provide additional bijections to complete
the proof of Theorem \ref{t-main}. Almost all of these bijections
are obtained with the help of our earlier bijections $\Psi$ and $\Phi_{213,231}$,
e.g., by composing them with other maps or restricting them appropriately.
For example, we will exploit reverse-complementation symmetry to obtain
a bijection for the case $\Pi=\{132,312\}$.
\item In Section \ref{s-132}, we will complete our bijective proof of Theorem
\ref{t-derdes} by giving a bijection from $\mathcal{D}_{n}(132)$
to $\mathfrak{D}_{n}(132)$. Our bijection is obtained by composing
three bijections, two of which are already known: one is due to Krattenthaler
\cite{Krattenthaler2001} and the other to Elizalde and Deutsch \cite{Elizalde2003}.
All of these bijections involve Dyck paths.
\item We end in Section \ref{s-open} with a couple of open problems relating
to this work.
\end{itemize}
We note that an independent proof of Conjecture \ref{cj-main} was
very recently provided by Dong and Xu \cite{Dong}. Given $\pi\in\mathfrak{S}_{n}$,
let $\ides(\pi)$ denote the number of descents of the inverse permutation
$\pi^{-1}$. Dong and Xu proved that $(\fix,\des)$ and $(\pix,\ides)$
are equidistributed over $\mathfrak{S}_{n}(\Pi)$ for all $\Pi$ listed
in Conjecture \ref{cj-main}. Because $\des(\pi)=\ides(\pi)$ for
all $\pi\in\mathfrak{S}_{n}(\sigma)$ where $\sigma\in\{132,213,231,312\}$
\cite[Corollary 3.2]{Stump2008/09}, their result is equivalent to
our Theorem \ref{t-main}. On the other hand, while $(\fix,\des)$
and $(\pix,\ides)$ are equidistributed over the full symmetric group
$\mathfrak{S}_{n}$ \cite[Theorem 1.2]{Foata2008a}, the same is not
true for $(\fix,\des)$ and $(\pix,\des)$.

While there is slight overlap in the ideas used in some of our proofs,
our approaches are fundamentally different yet complementary: Rather
than constructing explicit bijections, Dong and Xu prove the equidistributions
by showing that the corresponding generating functions are equal.
This generating function approach is advantageous in that it yields
formulas for computing the distributions, whereas our bijective approach
adds structural insight.

\section{\label{s-213-231}Proof of Theorem \ref{t-main} for \texorpdfstring{$\{213,231\}$}{\{213,231\}}-avoiding
permutations}

\subsection{Preliminaries on \texorpdfstring{$\{213,231\}$}{\{213,231\}}-avoiding
permutations}

Our proof of Theorem \ref{t-main} for almost all of the listed pattern
sets $\Pi$ will rely on a bijection $\Psi\colon\mathcal{W}_{n}\rightarrow\mathcal{W}_{n}$
sending $AD$-words of $\{213,231\}$-avoiding permutations with a
given number of fixed points to $AD$-words of $\{213,231\}$-avoiding
permutations with that number of pixed points. This bijection will
also preserve the number of $D$s. We begin by establishing some basic
facts about $\{213,231\}$-avoiding permutations.
\begin{lem}
\label{l-213-231-ad}Let $\pi=\pi_{1}\pi_{2}\cdots\pi_{n}\in\mathfrak{S}_{n}(213,231)$.
\begin{enumerate}
\item [\normalfont{(a)}] If $i\in[n-1]$ is the $j$th ascent of $\pi$,
then $\pi_{i}=j$.
\item [\normalfont{(b)}] If $i\in[n-1]$ is the $j$th descent of $\pi$,
then $\pi_{i}=n-j+1$.
\item [\normalfont{(c)}] If $\asc(\pi)=k$, then $\pi_{n}=k+1$.
\end{enumerate}
\end{lem}

In other words, every permutation in $\pi\in\mathfrak{S}_{n}(213,231)$
is the shuffle of an increasing sequence $12\cdots k$ (corresponding
to the ascents) with a decreasing sequence $n(n-1)\cdots(k+2)$ (corresponding
to the descents), along with the letter $k+1$ at the end.
\begin{proof}
We only prove (a). The proof of (b) is very similar to that of (a),
and (c) follows immediately from (a) and (b).

We proceed via induction on $j$. First consider the case $j=1$,
and let $i$ be the first ascent of $\pi$. Notice that the letter
$1$ cannot appear before position $i$, or else there would an earlier
ascent. Furthermore, we cannot have $\pi_{i+1}=1$ because $i$ is
an ascent, and $\pi_{i}\pi_{i+1}1$ would be an occurrence of $231$
if $1$ appears after $\pi_{i+1}$. Hence, we must have $\pi_{i}=1$.

Now, assume that the result holds for all $j$ up to $k$, and let
$i$ be the $(k+1)$th ascent of $\pi$. By the induction hypothesis,
$\pi_{i}$ cannot be less than $k+1$, as every letter from $1$ to
$k$ must appear earlier. Now, assume for contradiction that $\pi_{i}>k+1$.
Let $l$ be the $k$th ascent of $\pi$, so that $\pi_{l}=k$. Then
we have three cases for the placement of $k+1$ in $\pi$:
\begin{itemize}
\item If $k+1$ appears before position $l$ in $\pi$, then because $l$
is an ascent, we would have an occurrence $(k+1)k\pi_{l+1}$ of $213$
in $\pi$.
\item If $k+1$ appears strictly between positions $l$ and $i$ in $\pi$,
then we would have an additional ascent between the $k$th and $(k+1)$th
ascents of $\pi$.
\item If $k+1$ appears after position $i$ in $\pi$, then we would have
an occurrence $\pi_{i}\pi_{i+1}(k+1)$ of $231$ in $\pi$. 
\end{itemize}
We conclude that $\pi_{i}=k+1$, completing the induction.
\end{proof}
Recall that $\ad\colon\mathfrak{S}_{n}\rightarrow\mathcal{W}_{n-1}$
sends a permutation to its $AD$-word. As a consequence of Lemma \ref{l-213-231-ad},
$\ad$ restricts to a bijection between $\mathfrak{S}_{n}(213,231)$
and $\mathcal{W}_{n-1}$. We will use $\theta_{213,231}$ to denote
the inverse of this bijection, so that $\theta_{213,231}(w)$ is the
unique permutation in $\mathfrak{S}_{n}(213,231)$ with $AD$-word
$w$. For example, if $w=DDADAAAD$, then $\theta_{213,231}(w)=981723465$.

Given a word $w$ over $\{A,D\}$, let $\fix(w)\coloneqq\fix(\theta_{213,231}(w))$
and $\pix(w)\coloneqq\pix(\pi)$ where $\pi$ is any permutation with
$AD$-word $w$.\footnote{This definition makes sense because $\pix(\pi)$ depends only on the
underlying $AD$-word of $\pi$.} Then we would like our bijection $\Psi\colon\mathcal{W}_{n}\rightarrow\mathcal{W}_{n}$
to satisfy
\[
\fix(w)=\pix(\Phi(w))\quad\text{and}\quad\des(w)=\des(\Phi(w))
\]
for all $w\in\mathcal{W}_{n}$.

Before proceeding, we prove the following lemma concerning fixed points
of $\{213,231\}$-avoiding permutations. Here, $a_{i}(\pi)$ is the
number of $A$s and $d_{i}(\pi)$ the number of $D$s in the $i$th
prefix $w_{1}w_{2}\cdots w_{i}$ of the $AD$-word $\ad(\pi)=w_{1}w_{2}\cdots w_{n-1}$.
For instance, if $\pi=18765243$, then $\ad(\pi)=ADDDDAD$, so $a_{4}(\pi)=1$
and $d_{4}(\pi)=3$.
\begin{lem}
\label{l-213-231-fix}Let $n\geq2$ and $\pi\in\mathfrak{S}_{n}(213,231)$.
\begin{enumerate}
\item [\normalfont{(a)}] Suppose that $i\in[n-1]$ is an ascent of $\pi$.
Then $i$ is a fixed point of $\pi$ if and only if $d_{i}(\pi)=0$.
\item [\normalfont{(b)}] Suppose that $i\in[n-1]$ is a descent of $\pi$.
Then $i$ is a fixed point of $\pi$ if and only if $a_{i}(\pi)+2d_{i}(\pi)=n+1$.
\item [\normalfont{(c)}] The index $n$ is a fixed point of $\pi$ if and
only if $\pi=12\cdots n$.
\item [\normalfont{(d)}] The permutation $\pi$ has at most one fixed point
which is a descent.
\item [\normalfont{(e)}] The permutation $\pi$ is a derangement if and
only if 1 is a descent of $\pi$ and there does not exist $2\leq i\leq n-1$
such that $i$ is a descent of $\pi$ and $a_{i}(\pi)+2d_{i}(\pi)=n+1$.
\end{enumerate}
\end{lem}

\begin{proof}
We first prove (a). Suppose that $i$ is the $j$th ascent of $\pi$,
and that it is a fixed point of $\pi$. By Lemma \ref{l-213-231-ad}
(a), we have that $\pi_{i}=j$, so $i=j$. This implies that every
position before $i$ is also an ascent of $\pi$, so $d_{i}(\pi)=0$.
The converse follows similarly from Lemma \ref{l-213-231-ad} (a).

For (b), suppose that $i$ is the $j$th descent of $\pi$ (so that
$d_{i}(\pi)=j$), and that it is a fixed point of $\pi$. Then $\pi_{i}=n-j+1$
by Lemma \ref{l-213-231-ad} (b), so $n-j+1=i$. Substituting in $i=a_{i}(\pi)+d_{i}(\pi)$
and $j=d_{i}(\pi)$ yields $a_{i}(\pi)+2d_{i}(\pi)=n+1$. The converse
is similar.

Part (c) is immediate from $\pi$ avoiding $213$.

For (d), suppose that $i$ and $j$ are both descents and fixed points
of $\pi$, and assume $i\leq j$ without loss of generality. Then,
by (b), we have that $a_{i}(\pi)+2d_{i}(\pi)=a_{j}(\pi)+2d_{j}(\pi)$.
Moreover, $a_{i}(\pi)\leq a_{j}(\pi)$ and $d_{i}(\pi)\leq d_{j}(\pi)$,
which together with $a_{i}(\pi)+2d_{i}(\pi)=a_{j}(\pi)+2d_{j}(\pi)$
forces $a_{i}(\pi)=a_{j}(\pi)$ and $d_{i}(\pi)=d_{j}(\pi)$. This
implies $i=j$, and the result follows.

Finally, (e) is an immediate consequence of (a)\textendash (c).
\end{proof}

\subsection{From \texorpdfstring{$\{213,231\}$}{\{213,231\}}-avoiding derangements
to \texorpdfstring{$\{213,231\}$}{\{213,231\}}-avoiding desarrangements}

To construct the bijection $\Psi$, we will first give a bijection
between the underlying $AD$-words of $\{213,231\}$-avoiding derangements
and $\{213,231\}$-avoiding desarrangements. To that end, let us define
the appropriate sets of $AD$-words. Below, $a_{i}(w)$ is the number
of $A$s and $d_{i}(w)$ the number of $D$s in the $i$th prefix
$w_{1}w_{2}\cdots w_{i}$ of the word $w=w_{1}w_{2}\cdots w_{n}$.
\begin{itemize}
\item Let $\mathcal{W}_{n}^{(f)}$ denote the set of words $w=w_{1}w_{2}\cdots w_{n}\in\mathcal{W}_{n}$
where $w_{1}=D$ and there does not exist $2\leq i\leq n$ such that
$w_{i}=D$ and $a_{i}(w)+2d_{i}(w)=n+2$. By Lemma \ref{l-213-231-fix}
(e), $\mathcal{W}_{n-1}^{(f)}$ consists of precisely the $AD$-words
of permutations in $\mathcal{D}_{n}(213,231)$.
\item Let $\mathcal{W}_{n}^{(p)}$ consist of all words in $\mathcal{W}_{n}$
where the first $A$ is in an even position, along with the word $D^{n}$
of all $D$s if and only if $n$ is odd. Then the words in $\mathcal{W}_{n-1}^{(p)}$
are precisely the $AD$-words of length $n$ desarrangements.
\end{itemize}
So, we will define a bijection $\psi\colon\mathcal{W}_{n}^{(f)}\rightarrow\mathcal{W}_{n}^{(p)}$.
For this purpose, we shall need to associate each of $\mathcal{W}_{n}^{(f)}$
and $\mathcal{W}_{n}^{(p)}$ with a set of integer compositions.
\begin{lem}
\label{l-dercomp}Let $L=(L_{1},L_{2},\dots,L_{n})$ be a composition
with all parts $1$ or $2$. Then the following are equivalent\textup{:}
\begin{enumerate}
\item [\normalfont{(i)}] There exists a \textup{(}unique\textup{)}
$i\in[n]$ such that $\sum_{j=1}^{i}L_{j}=n$.
\item [\normalfont{(ii)}] There does not exist an $i\in[n]$ such that
$L_{i}=2$ and $\sum_{j=1}^{i}L_{j}=n+1$.
\item [\normalfont{(iii)}] There exists a \textup{(}unique\textup{)}
$i\in[n]$ such that the number of 2s in the prefix $(L_{1},L_{2},\dots,L_{i})$
is equal to $n-i$, the total number of parts of the suffix $(L_{i+1},L_{i+2},\dots,L_{n})$.
\end{enumerate}
Furthermore, the unique $i$ in conditions \textup{(i)} and \textup{(iii)}
are the same.

\end{lem}

\begin{proof}
The equivalence between (i) and (ii) is readily seen. Moreover, taking
$t(i)$ to be the number of 2s in $(L_{1},L_{2},\dots,L_{i})$, the
equivalence of (i) and (iii) is a straightforward consequence of the
observation that $t(i)+i=\sum_{j=1}^{i}L_{j}$. We omit the details.
\end{proof}
Let $\Comp_{n}^{(f)}$ be the set of compositions satisfying the equivalent
conditions in Lemma~\ref{l-dercomp}. (And by convention, let $\Comp_{0}^{(f)}$
consist of the empty composition.) For instance, we have 
\[
\Comp_{3}^{(f)}=\{(1,1,1),(2,1,1),(2,1,2),(1,2,1),(1,2,2)\}.
\]

\begin{prop}
\label{p-derbij}For all $n\geq1$, the sets $\mathcal{W}_{n}^{(f)}$
and $\Comp_{n-1}^{(f)}$ are in bijection.
\end{prop}

\begin{proof}
Let $w\in\mathcal{W}_{n}^{(f)}$. We obtain the composition $\gamma(w)=(L_{1},L_{2},\dots,L_{n-1})$
by taking the $n-1$ non-initial letters of $w$, replacing each $A$
with a 1, and replacing each $D$ with a 2:
\[
L_{i}=\begin{cases}
1, & \text{if }w_{i+1}=A,\\
2, & \text{if }w_{i+1}=D,
\end{cases}
\]
for all $i\in[n-1]$. Translating the definition of $\mathcal{W}_{n}^{(f)}$
in terms of the composition $\gamma(w)$, we see that there does not
exist an $i\in[n-1]$ such that $L_{i}=2$ and $\sum_{j=1}^{i}L_{j}=n$.
Hence, $\gamma(w)\in\Comp_{n-1}^{(f)}$, and it is clear that $\gamma\colon\mathcal{W}_{n}^{(f)}\rightarrow\Comp_{n-1}^{(f)}$
is a bijection.
\end{proof}
Let $\Comp_{n}^{(p)}$ denote the set of compositions of $n$ into
one type of 1 and two types of 2, denoted $2$ and $\bar{2}$. For
example, we have 
\[
\Comp_{3}^{(p)}=\{(1,1,1),(2,1),(\bar{2},1),(1,2),(1,\bar{2})\}.
\]
Be aware that the $n$ in the notation $\Comp_{n}^{(f)}$ refers to
the number of parts, whereas in $\Comp_{n}^{(p)}$ it refers to the
sum of the parts.
\begin{prop}
\label{p-213-231-derjac}For all $n\geq0$, the sets $\Comp_{n}^{(f)}$
and $\Comp_{n}^{(p)}$ are in bijection.
\end{prop}

\begin{proof}
Let $L\in\Comp_{n}^{(f)}$, so that $\sum_{j=1}^{i}L_{j}=n$ for some
$i\in[n]$. Define $J=(L_{1},L_{2},\dots,L_{i})$ and $K=(L_{i+1},L_{i+2},\dots,L_{n})$;
then the number of 2s in $J$ is equal to the number of parts of $K$.
We will now use the composition $K$ as a ``key'' to decorate a subset
of the 2s of $J$ with bars. To do so, iterate through the parts of
$K$ from left to right. If the $j$th part of $K$ is a 2, then decorate
the $j$th 2 of $J$ (from left to right) with a bar. The result\textemdash denote
it $\lambda(L)$\textemdash is a composition in $\Comp_{n}^{(p)}$.

For example, if $L=(2,2,2,1,1,2,1,2,2,1,2)$, then we have $J=(2,2,2,1,1,2,1)$
and $K=(2,2,1,2)$. Using $K$ to decorate $J$, we arrive at $\lambda(L)=(\bar{2},\bar{2},2,1,1,\bar{2},1)$.

It is readily seen that $\lambda\colon\Comp_{n}^{(f)}\rightarrow\Comp_{n}^{(p)}$
is a bijection. In short, given $M\in\Comp_{n}^{(p)}$, we can use
the bars to determine $K$, erase the bars from $M$ to obtain $J$,
and then concatenate $J$ and $K$ to obtain $\lambda^{-1}(M)$.
\end{proof}
\begin{prop}
\label{p-desbij}For all $n\geq1$, the sets $\Comp_{n-1}^{(p)}$
and $\mathcal{W}_{n}^{(p)}$ are in bijection.
\end{prop}

\begin{proof}
Traverse through the parts of $L\in\Comp_{n-1}^{(p)}$ from left to
right, replacing each $1$ with $A$, each 2 with $AD$, and each
$\bar{2}$ with $DD$. Next, preprend a $D$ to this concatenation,
which results in a word $\eta(L)$ of length $n$. Notice that the
first $A$ in $\eta(L)$ appears when the first 1 or 2 appears in
$L$, and since each $\bar{2}$ is mapped to two letters and we have
prepended a $D$, there is an odd number of $D$s before the first
$A$ (or $\eta(L)$ consists only of an odd number of $D$s). Hence,
$\eta(L)\in\mathcal{W}_{n}^{(p)}$, and $\eta$ is clearly a bijection.
\end{proof}
Thus $\psi\colon\mathcal{W}_{n}^{(f)}\rightarrow\mathcal{W}_{n}^{(p)}$
defined by 
\[
\psi=\eta\circ\lambda\circ\gamma
\]
is a bijection. Moreover, the composition $\theta_{213,231}\circ\psi\circ\ad$
is a bijection from $\mathcal{D}_{n}(213,231)$ to $\mathfrak{D}_{n}(213,231)$.
\begin{example}
Take $\pi=13\,12\,11\,10\,1\,2\,9\,3\,8\,7\,4\,6\,5\in\mathcal{D}_{13}(213,231)$.
Then 
\[
\ad(\pi)=DDDDAADADDAD,\quad\text{so}\quad\gamma(\ad(\pi))=(2,2,2,1,1,2,1,2,2,1,2).
\]
As shown in the proof of Proposition \ref{p-213-231-derjac}, we have
$\lambda(\gamma(\ad(\pi)))=(\bar{2},\bar{2},2,1,1,\bar{2},1)$, whence
\[
\psi(\ad(\pi))=\eta(\lambda(\gamma(\ad(\pi))))=DDDDDADAADDA.
\]
Then 
\[
\theta_{213,231}(DDDDDADAADDA)=13\,12\,11\,10\,9\,1\,8\,2\,3\,7\,6\,4\,5
\]
is the unique permutation in $\mathfrak{D}_{13}(213,231)$ with $AD$-word
$DDDDDADAADDA$.
\end{example}

By the construction of $\psi$, the words $w$ and $\psi(w)$ have
the same number of $D$s for all $w\in\mathcal{W}_{n}^{(f)}$. Hence,
$\theta_{213,231}\circ\psi\circ\ad$ is a $\des$-preserving bijection
from $\mathcal{D}_{n}(213,231)$ to $\mathfrak{D}_{n}(213,231)$.

\subsection{An auxiliary bijection}

We seek to extend the bijection $\psi$ to our desired bijection $\Psi$.
Let us first describe the basic idea of this extension, which is similar
to that of extending D\'{e}sarm\'{e}nien's bijection $\Phi$ to
the full symmetric group $\mathfrak{S}_{n}$. 

Take the $AD$-word of a permutation in $\mathfrak{S}_{n}(213,231)$,
and remove its initial run of $A$s. The result will be the $AD$-word
of a $\{213,231\}$-avoiding derangement, or of a $\{213,231\}$-avoiding
permutation with a single fixed point that is a descent. In the first
case, we apply the earlier bijection $\psi$ and prepend the initial
run of $A$s that was removed earlier. The second case will require
an auxiliary bijection $\hat{\psi}$, involving the following sets
of words:
\begin{itemize}
\item Let $\widehat{\mathcal{W}}_{n}^{(f)}$ denote the set of words $w\in\mathcal{W}_{n}$
where $w_{1}=D$ and there exists $2\leq i\leq n$ such that $w_{i}=D$
and $a_{i}(w)+2d_{i}(w)=n+2$. By Lemma \ref{l-213-231-fix}, $\widehat{\mathcal{W}}_{n-1}^{(f)}$
consists of precisely the $AD$-words of permutations in $\mathfrak{S}_{n}(213,231)$
whose only fixed point is a descent.
\item Let $\widehat{\mathcal{W}}_{n}^{(p)}\coloneqq\{\,Dw:w\in\mathcal{W}_{n-1}^{(p)}\,\}$,
so that the words in $\widehat{\mathcal{W}}_{n-1}^{(p)}$ are precisely
the $AD$-words of length $n$ permutations whose only pixed point
is a descent.
\end{itemize}
Notice that any $w\in\mathcal{W}_{n}$ beginning with $D$ is either
in $\mathcal{W}_{n}^{(f)}$ or $\widehat{\mathcal{W}}_{n}^{(f)}$,
and is either in $\mathcal{W}_{n}^{(p)}$ or $\widehat{\mathcal{W}}_{n}^{(p)}$.
This fact will be used implicitly in Lemmas \ref{l-fixw}\textendash \ref{l-pixw}.
\begin{prop}
For all $n\geq2$, the sets $\widehat{\mathcal{W}}_{n}^{(f)}$ and
$\Comp_{n-2}^{(f)}$ are in bijection.
\end{prop}

We prove this proposition by giving a variant of the bijection $\gamma\colon\mathcal{W}_{n}^{(f)}\rightarrow\Comp_{n-1}^{(f)}$
defined in the proof of Proposition \ref{p-derbij}.
\begin{proof}
Let $w\in\widehat{\mathcal{W}}_{n}^{(f)}$. Define $L=(L_{1},L_{2},\dots,L_{n-1})$
by 
\[
L_{i}=\begin{cases}
1, & \text{if }w_{i+1}=A,\\
2, & \text{if }w_{i+1}=D,
\end{cases}
\]
for all $i\in[n-1]$. Translating the definition of $\widehat{\mathcal{W}}_{n}^{(f)}$
in terms of $L$, there exists a unique $i\in[n-1]$ such that $L_{i}=2$
and $\sum_{j=1}^{i}L_{j}=n$. Now, let $\hat{\gamma}(w)$ be the composition
obtained by deleting the part $L_{i}$ from $L$. Then $\sum_{j=1}^{i-1}L_{j}=n-2$,
which means that $\hat{\gamma}(w)\in\Comp_{n-2}^{(f)}$. 

Every step of the above procedure is reversible. In particular, given
a composition $M=(M_{1},M_{2},\dots,M_{n-2})\in\Comp_{n-2}^{(f)}$,
we know there is a unique $i\in[n-2]$ such that $\sum_{j=1}^{i}M_{j}=n$,
so we insert a 2 immediately after $M_{i}$ before converting the
resulting composition back into a word. Therefore, $\hat{\gamma}\colon\widehat{\mathcal{W}}_{n}^{(f)}\rightarrow\Comp_{n-2}^{(f)}$
is a bijection.
\end{proof}
\begin{prop}
For all $n\geq2$, the sets $\Comp_{n-2}^{(p)}$ and $\widehat{\mathcal{W}}_{n}^{(p)}$
are in bijection.
\end{prop}

\begin{proof}
Define $\hat{\eta}\colon\Comp_{n-2}^{(p)}\rightarrow\widehat{\mathcal{W}}_{n}^{(p)}$
in the same way as $\eta$ in the proof of Proposition \ref{p-desbij}
except that we prepend $DD$ instead of $D$. Then $\hat{\eta}$ is
clearly a bijection.
\end{proof}
Let
\[
\hat{\psi}\coloneqq\hat{\eta}\circ\lambda\circ\hat{\gamma}
\]
be our desired bijection from $\widehat{\mathcal{W}}_{n}^{(f)}$ to
$\mathcal{\widehat{\mathcal{W}}}_{n}^{(p)}$. Like $\psi$, this bijection
preserves the number of $D$s.
\begin{example}
\label{eg-hatbij}Take $w=DADDADAD\in\widehat{\mathcal{W}}_{8}^{(f)}$.
Replacing each $A$ with 1 and each (non-initial) $D$ with 2 yields
\[
L=(1,2,2,1,2,1,2).
\]
Because $1+2+2+1+2=9-1$, we remove the third $2$ from $L$ to obtain
\[
\hat{\gamma}(w)=(1,2,2,1,1,2).
\]
The composition $\hat{\gamma}(w)$ splits as $J=(1,2,2,1)$ and $K=(1,2)$,
and we use $K$ to decorate $J$:
\[
\lambda(\hat{\gamma}(w))=(1,2,\bar{2},1).
\]
This yields
\[
\hat{\psi}(w)=\hat{\eta}(\lambda(\hat{\gamma}(w)))=DDAADDDA\in\widehat{\mathcal{W}}_{8}^{(p)}.
\]
Note that $w=DADDADAD$ is the $AD$-word of $91872635\in\mathfrak{S}_{9}(213,231)$,
whose only fixed point is a descent. Similarly, $\hat{\psi}(w)=DDAADDDA$
is the $AD$-word of $981276534\in\mathfrak{S}_{9}(213,231)$, whose
only pixed point is a descent.
\end{example}

\subsection{The bijection \texorpdfstring{$\Psi$}{Psi}}

Now that we have the bijections $\psi$ and $\hat{\psi}$, we can
use these to construct $\Psi$. Notice that any nonempty word $w\in\mathcal{W}_{n}$
can be written as $w=A^{k}v$ where $v$ is the maximal suffix of
$w$ that does not begin with an $A$ (so $v$ either begins with
a $D$ or is empty). For example, if $w=AADADDD$, then $v=DADDD$. 

For the remainder of this section, whenever we write a nonempty word
$w\in\mathcal{W}_{n}$ as $w=A^{k}v$, the suffix $v$ is defined
as above.
\begin{lem}
\label{l-fixw}Let $n\geq1$. Then for any $w=A^{k}v\in\mathcal{W}_{n}$,
we have
\[
\fix(w)=k+\begin{cases}
1, & \text{if }v\in\widehat{\mathcal{W}}_{n-k}^{(f)}\text{ or }v\text{ is empty},\\
0, & \text{otherwise, if }v\in\mathcal{W}_{n-k}^{(f)}.
\end{cases}
\]
\end{lem}

\begin{proof}
We have the following cases:
\begin{itemize}
\item \textbf{Case 1: }Suppose that there exists $i\in[n]$ such that $w_{i}=D\text{ and }a_{i}(w)+2d_{i}(w)=n+1$.
By Lemma \ref{l-213-231-fix}, the length $n+1$ permutation $\theta_{213,231}(w)$
has $k$ ascent fixed points and a descent fixed point. Moreover,
$\theta_{213,231}(w)$ cannot be the identity permutation $12\cdots(n+1)$
because it has a descent, so $n+1$ is not a fixed point. Hence, $\fix(w)=k+1$
and we wish to show that $v\in\widehat{\mathcal{W}}_{n-k}^{(f)}$.

We know that $v$ is nonempty because $w$ contains a $D$ and $A^{k}$
does not; thus, $v_{1}=D$. Moreover, since $v$ is obtained by removing
the first $k$ letters of $w$\textemdash all of which are $A$s\textemdash it
follows that $i>k$, $v_{i-k}=D$, and $a_{i-k}(v)+k+2d_{i-k}(v)=n+1$.
Rearranging this equation yields $a_{i-k}(v)+2d_{i-k}(v)=n-k+1$,
so we indeed have $v\in\widehat{\mathcal{W}}_{n-k}^{(f)}$.
\item \textbf{Case 2: }Suppose that no such $i$ exists. Again by Lemma
\ref{l-213-231-fix}, the permutation $\theta_{213,231}(w)$ has $k$
ascent fixed points and no descent fixed points. If $\theta_{213,231}(w)$
is the identity permutation, we have $\fix(w)=k+1$ and $v$ is empty.
Otherwise, we have $\fix(w)=k$ and we wish to show that $v\in\mathcal{W}_{n-k}^{(f)}$.

Suppose instead that $v\notin\mathcal{W}_{n-k}^{(f)}$, so there exists
$j\in[n-k]$ such that $v_{j}=D$ and $a_{j}(v)+2d_{j}(v)=n-k+1$.
Similar to before, this is equivalent to $w_{j+k}=D$ and $a_{j+k}(w)+2d_{j+k}(w)=n+1$,
which is a contradiction. Therefore, $v\in\mathcal{W}_{n-k}^{(f)}$
as desired.\qedhere
\end{itemize}
\end{proof}
We now give a $\pix$ analogue of the preceding lemma.
\begin{lem}
\label{l-pixw}Let $n\geq1$. Then for any $w=A^{k}v\in\mathcal{W}_{n}$,
we have
\[
\pix(w)=k+\begin{cases}
1, & \text{if }v\in\widehat{\mathcal{W}}_{n-k}^{(p)}\text{ or }v\text{ is empty},\\
0, & \text{otherwise, if }v\in\mathcal{W}_{n-k}^{(p)}.
\end{cases}
\]
\end{lem}

\begin{proof}
Let $\pi$ be any permutation with $AD$-word $w$, and let $\pi=\iota\delta$
be its pixed factorization. Let $j=\pix(w)$, the length of $\iota$.
Then each index $1,2,\dots,j-1$ is an ascent of $\pi$, but $j$
may not necessarily be an ascent of $\pi$, and $j+1$ is a descent
of $\pi$ assuming that $\delta$ is nonempty. This means that either
$k=j-1$ or $k=j$. We consider the following cases:
\begin{itemize}
\item \textbf{Case 1:} Suppose that $j$ is an ascent of $\pi$. Then $\pix(w)=j=k$
and $v$ is the $AD$-word of $\delta$. Since $\delta$ is a desarrangement,
we have $v\in\mathcal{W}_{n-k}^{(p)}$.
\item \textbf{Case 2:} Suppose that $j$ is a descent of $\pi$. Then $\pix(w)=j=k+1$
and $v$ is the $AD$-word of $\iota_{j}\delta$ (the last letter
of $\iota$ prepended to the desarrangement $\delta$). Then $v=Du$
where $u$ is the $AD$-word of $\delta$. Again, $\delta$ is a desarrangement,
so $u\in\mathcal{W}_{n-k}^{(p)}$. Thus $v\in\widehat{\mathcal{W}}_{n-k}^{(p)}$. 
\item \textbf{Case 3:} If $j$ is neither an ascent nor descent of $\pi$,
then $\pix(w)=j=k+1$ and $v$ is empty.\qedhere
\end{itemize}
\end{proof}
Define $\Psi\colon\mathcal{W}_{n}\rightarrow\mathcal{W}_{n}$ by 
\[
\Psi(w)=\begin{cases}
A^{k}\psi(v), & \text{if }v\in\mathcal{W}_{n-k}^{(f)},\\
A^{k}\hat{\psi}(v), & \text{if }v\in\widehat{\mathcal{W}}_{n-k}^{(f)},\\
A^{k}, & \text{otherwise, if }v\text{ is empty},
\end{cases}
\]
for all $w=A^{k}v\in\mathcal{W}_{n}$. This is a bijection with inverse
\[
\Psi^{-1}(w)=\begin{cases}
A^{k}\psi^{-1}(v), & \text{if }v\in\mathcal{W}_{n-k}^{(p)},\\
A^{k}\hat{\psi}^{-1}(v), & \text{if }v\in\widehat{\mathcal{W}}_{n-k}^{(p)},\\
A^{k}, & \text{otherwise, if }v\text{ is empty.}
\end{cases}
\]
Lemmas \ref{l-fixw}\textendash \ref{l-pixw} guarantee that $\fix(w)=\pix(\Psi(w))$
for all $w\in\mathcal{W}_{n}$. And, as noted earlier, $\psi$ and
$\hat{\psi}$ both preserve the number of $D$s, so $\Psi$ does as
well. Thus
\[
\Phi_{213,231}\coloneqq\theta_{213,231}\circ\Psi\circ\ad
\]
satisfies $\fix(\pi)=\pix(\Phi_{213,231}(\pi))$ and $\des(\pi)=\des(\Phi_{213,231}(\pi))$
for all $\pi\in\mathfrak{S}_{n}(213,231)$, giving a bijective proof
of the $\Pi=\{213,231\}$ case of Theorem \ref{t-main}.
\begin{example}
\label{eg-213-231-bij}Take $\pi=1\,2\,11\,3\,10\,9\,4\,8\,5\,7\,6\in\mathfrak{S}_{11}(213,231)$,
which has 3 fixed points and 5 descents. We have 
\[
\ad(\pi)=AADADDADAD,
\]
and removing the initial run of $A$s yields $v=DADDADAD\in\widehat{\mathcal{W}}_{8}^{(f)}$.
In Example \ref{eg-hatbij}, we saw that 
\[
\hat{\psi}(v)=DDAADDDA\in\widehat{\mathcal{W}}_{8}^{(p)},
\]
so 
\[
\Psi(\ad(\pi))=AADDAADDDA.
\]
The unique permutation in $\mathfrak{S}_{11}(213,231)$ with this
$AD$-word is 
\[
\Phi_{213,231}(\pi)=\theta_{213,231}(AADDAADDDA)=1\,2\,11\,10\,3\,4\,9\,8\,7\,5\,6,
\]
which indeed has 3 pixed points and 5 descents.
\end{example}

\section{\label{s-others}Completing the proof of Theorem \ref{t-main}}

As of this point, we have verified Theorem \ref{t-main} only for
$\Pi=\{213,231\}$. We will now prove this theorem for all remaining
pattern sets $\Pi$ listed in its statement.

\subsection{The pattern set \texorpdfstring{$\Pi=\{132,312\}$}{Pi=\{132,312\}}}

Given a permutation $\pi=\pi_{1}\pi_{2}\cdots\pi_{n}$, the \textit{reverse}
of $\pi$ is defined to be $r(\pi)\coloneqq\pi_{n}\cdots\pi_{2}\pi_{1}$,
the \textit{complement} $c(\pi)$ of $\pi$ to be the permutation
obtained by (simultaneously) replacing the $i$th smallest letter
of $\pi$ with the $i$th largest letter of $\pi$ for all $i\in[n]$,
and the\textit{ reverse-complement} of $\pi$ to be $rc(\pi)\coloneqq c(r(\pi))=r(c(\pi))$.
For instance, if $\pi=162354$, then we have $r(\pi)=453261$, $c(\pi)=615423$,
and $rc(\pi)=324516$.

Reverse-complementation in particular has several useful properties
related to pattern avoidance and the permutation statistics under
consideration:
\begin{enumerate}
\item Given a pattern set $\Pi$, let $\Pi^{rc}\coloneqq\{\,rc(\pi):\pi\in\Pi\,\}$.
Then reverse-complementation restricts to a bijection from $\mathfrak{S}_{n}(\Pi)$
to $\mathfrak{S}_{n}(\Pi^{rc})$.
\item If $w=w_{1}w_{2}\cdots w_{n-1}$ is the $AD$-word of $\pi\in\mathfrak{S}_{n}$,
then $r(w)\coloneqq w_{n-1}\cdots w_{2}w_{1}$ is the $AD$-word of
the reverse-complement $rc(\pi)$. Hence, reverse-complementation
preserves the number of descents. 
\item Reverse-complementation is the same as conjugation by the decreasing
permutation $n\cdots21$, so it preserves cycle type and thus the
number of fixed points.
\end{enumerate}
We can exploit these properties, along with our earlier work for $\Pi=\{213,231\}$,
to prove our desired equidistribution for $\Pi^{rc}=\{132,312\}$.

Recall that $\ad$ restricts to a bijection between $\mathfrak{S}_{n}(213,231)$
and $\mathcal{W}_{n-1}$, so that every word in $\mathcal{W}_{n-1}$
is the $AD$-word of exactly one permutation in $\mathfrak{S}_{n}(213,231)$.
And since reverse-complementation induces a bijection on $AD$-words,
this means that $\ad$ also restricts to a bijection between $\mathfrak{S}_{n}(132,312)$
and $\mathcal{W}_{n-1}$. Let the inverse of this bijection be denoted
$\theta_{132,312}$.

It follows from the above discussion, along with properties of $\Psi$
established earlier, that
\[
\Phi_{132,312}\coloneqq\theta_{132,312}\circ\Psi\circ\ad\circ\,rc
\]
is a bijection on $\mathfrak{S}_{n}(132,312)$ satisfying $\fix(\pi)=\pix(\Phi_{132,312}(\pi))$
and $\des(\pi)=\des(\Phi_{132,312}(\pi))$ for all $\pi\in\mathfrak{S}_{n}(132,312)$.
Thus we have proven Theorem \ref{t-main} for $\Pi=\{132,312\}$.
\begin{example}
Take $\pi=6\,5\,7\,4\,8\,3\,2\,9\,1\,10\,11\in\mathfrak{S}_{11}(132,312)$,
which has 3 fixed points and 5 descents. The reverse-complement of
$\pi$ is 
\[
rc(\pi)=1\,2\,11\,3\,10\,9\,4\,8\,5\,7\,6\in\mathfrak{S}_{11}(213,231),
\]
and in Example \ref{eg-213-231-bij}, we saw that 
\[
\Psi(\ad(1\,2\,11\,3\,10\,9\,4\,8\,5\,7\,6))=AADDAADDDA.
\]
Then 
\[
\Phi_{132,312}(\pi)=\theta_{132,312}(AADDAADDDA)=6\,7\,8\,5\,4\,9\,10\,3\,2\,1\,11
\]
is a permutation in $\mathfrak{S}_{11}(132,312)$ with 3 pixed points
and 5 descents, as expected.
\end{example}

\subsection{The pattern set \texorpdfstring{$\Pi=\{132,321\}$}{Pi=\{132,321\}}}

Next, we tackle the case $\Pi=\{132,321\}$, which will not require
the bijection $\Psi$.

An \textit{interval} of a permutation $\pi$ is a consecutive subsequence
of $\pi$ consisting of consecutive integer letters. For example,
$234$ is an increasing interval of $234165$, whereas $2341$ is
a non-increasing interval. The permutations in $\mathfrak{S}_{n}(132,321)$
can be characterized as concatenations of the increasing intervals
$x(x+1)\cdots y$, $12\cdots(x-1)$, and $(y+1)(y+2)\cdots n$ in
which $1\leq x\leq y\leq n$.
\begin{lem}
\label{l-132-321}Let $n\geq1$. Then 
\[
\mathfrak{S}_{n}(132,321)=\{\,x(x+1)\cdots y12\cdots(x-1)(y+1)(y+2)\cdots n:1\leq x\leq y\leq n\,\}.
\]
\end{lem}

Note that $x(x+1)\cdots y12\cdots(x-1)(y+1)(y+2)\cdots n$ is the
identity permutation $12\cdots n$ when $x=1$, and may not actually
end with the letter $n$ (consider the case $y=n$).
\begin{proof}
Let $\pi\in\mathfrak{S}_{n}(132,321)$. Then $\pi=\tau1\rho$ where
$\tau$ and $\rho$ are both increasing permutations. If $\tau$ is
empty, then $\pi=12\cdots n$. If $\tau$ is nonempty, then it must
be an interval; otherwise, $\pi$ would contain an occurrence of $132$
consisting of two letters of $\tau$ and one letter of $\rho$. Thus,
$\pi$ is of the desired form.
\end{proof}
\begin{lem}
\label{l-132-321-fixpix}Let $n\geq2$ and $0\leq k\leq n-2$. Then\textup{:}
\begin{enumerate}
\item [\normalfont{(a)}]The permutations in $\mathfrak{S}_{n}(132,321)$
with exactly $k$ fixed points are precisely those of the form 
\[
x(x+1)\cdots(n-k)12\cdots(x-1)(n-k+1)(n-k+2)\cdots n
\]
where $2\leq x\leq n-k$, all of which have exactly 1 descent.
\item [\normalfont{(b)}]The permutations in $\mathfrak{S}_{n}(132,321)$
with exactly $k$ pixed points are precisely those of the form 
\[
x(x+1)\cdots(x+k)12\cdots(x-1)(x+k+1)(x+k+2)\cdots n
\]
where $2\leq x\leq n-k$, all of which have exactly 1 descent.
\end{enumerate}
\end{lem}

\begin{proof}
Let $\pi\in\mathfrak{S}_{n}(132,321)$. By Lemma \ref{l-132-321},
we have that 
\[
\pi=x(x+1)\cdots y12\cdots(x-1)(y+1)(y+2)\cdots n
\]
 for some $1\leq x\leq y\leq n$. Assume that $x\neq1$, so that $\pi$
is not the identity permutation. Observe that $\pi$ has exactly 1
descent, which occurs at the end of the increasing interval $x(x+1)\cdots(x+k)$.

The fixed points of $\pi$ are precisely the letters in the increasing
interval $(y+1)(y+2)\cdots n$. So for $\fix(\pi)=k$, this interval
must be of length $k$, which forces $y=n-k$. Similarly, the pixed
points of $\pi$ are precisely the letters in the increasing interval
$x(x+1)\cdots y$ except for the last letter $y$, so $\pix(\pi)=k$
would force $y=x+k$.
\end{proof}
Lemma \ref{l-132-321-fixpix} tells us that, given $0\leq k\leq n-2$
and $2\leq x\leq n-k$, there is a unique permutation in $\mathfrak{S}_{n}(132,321)$
with $k$ fixed points and first letter $x$, as well as a unique
permutation $\mathfrak{S}_{n}(132,321)$ with $k$ pixed points and
first letter $x$. All of these permutations have a single descent.
The only other permutation in $\mathfrak{S}_{n}(132,321)$ is the
identity permutation $12\cdots n$, which has $k$ fixed points and
$k$ pixed points. Hence, we define
\[
\Phi_{132,312}\colon\mathfrak{S}_{n}(132,321)\rightarrow\mathfrak{S}_{n}(132,321)
\]
by taking $\Phi_{132,312}(\pi)$ to be the unique permutation in $\mathfrak{S}_{n}(132,321)$
whose number of pixed points is equal to the number of fixed points
of $\pi$ and with the same first letter as $\pi$. This immediately
yields Theorem \ref{t-main} for $\Pi=\{132,312\}$.

\subsection{Tripleton pattern sets}

Among the $\Pi$ listed in Theorem \ref{t-main}, only those of size
3 remain. As we shall demonstrate, our bijections for all but one
of these are obtained by restricting our bijections $\Phi_{213,231}$
and $\Phi_{132,312}$ appropriately.
\begin{lem}
\label{l-213-231-restr}Let $n\geq1$, and suppose that $\pi\in\mathfrak{S}_{n}(213,231)$.
Then\textup{:}
\begin{enumerate}
\item [\normalfont{(a)}] $\pi$ avoids $123$ if and only if $\asc(\pi)\leq1$,
\item [\normalfont{(b)}] $\pi$ avoids $321$ if and only if $\des(\pi)\leq1$,
and
\item [\normalfont{(c)}] $\pi$ avoids $312$ if and only if $\ad(\pi)=A^{k}D^{n-k-1}$
where $k=\asc(\pi)$.
\end{enumerate}
\end{lem}

\begin{proof}
Assume that $\asc(\pi)\geq2$, and let $i<j$ be the two smallest
ascents of $\pi$. Then $\pi_{i}=1$ and $\pi_{j}=2$ by Lemma \ref{l-213-231-ad},
so $\pi_{i}\pi_{j}\pi_{j+1}$ is an occurrence of $123$ in $\pi$.
Conversely, suppose that $\asc(\pi)\leq1$. Then Lemma \ref{l-213-231-ad}
implies that either $\pi=n\cdots21$ (if $\pi$ has no ascents) or
$\pi=1n\cdots32$ (if $\pi$ has exactly 1 ascent), both of which
avoid $123$. Thus, (a) follows.

The proof of (b) is omitted due to its similarity to that of (a).

For (c), suppose to the contrary that $\pi$ avoids $312$ but that
$\ad(\pi)$ is not of the stated form. Then $\pi$ has a descent $i$
followed immediately by an ascent $i+1$. Consequently, $\pi_{i}\pi_{i+1}\pi_{i+2}$
must be an occurrence of either $213$ or $312$, a contradiction.
Conversely, if $\ad(\pi)=A^{k}D^{n-k-1}$ where $k=\asc(\pi)$, then
it is immediate that $\pi$ avoids $312$.
\end{proof}
We are now ready to prove the equidistribution of $(\fix,\des)$ and
$(\pix,\des)$ over the avoidance classes $\mathfrak{S}_{n}(123,213,231)$,
$\mathfrak{S}_{n}(213,231,312)$, and $\mathfrak{S}_{n}(213,231,321)$.
\begin{proof}[Proof of Theorem \ref{t-main} for $\{123,213,231\}$, $\{213,231,312\}$,
$\{213,231,321\}$.]
We already know that $\Phi_{213,231}$ sends $(\fix,\des)$ to $(\pix,\des)$,
and hence preserves the number of ascents as well. Thus, by Lemma
\ref{l-213-231-restr} (a), $\Phi_{213,231}$ restricts to a bijection
on $\mathfrak{S}_{n}(123,213,231)$, which proves the equidistribution
for this avoidance class. The analogous result for $\mathfrak{S}_{n}(213,231,321)$
follows by the same reasoning, using Lemma \ref{l-213-231-restr}
(b) instead.

Now, let $\pi\in\mathfrak{S}_{n}(213,231,312)$. By Lemma \ref{l-213-231-restr}
(c), $\ad(\pi)=A^{k}D^{n-k-1}$ where $k=\asc(\pi)$. Then $\Phi(A^{k}D^{n-k-1})$
is either $A^{k}$, $A^{k}\phi(D^{n-k-1})$, or $A^{k}\hat{\phi}(D^{n-k-1})$,
depending on the parity of $n-k-1$ and whether $n=k+1$. But we have
$\Psi(A^{k}D^{n-k-1})=A^{k}D^{n-k-1}$ in all of these cases; this
can be seen directly from the definitions of $\phi$ and $\hat{\phi}$,
or from the fact that $\Phi$ preserves the number of $D$s. Consequently,
$\Phi_{213,231}$ restricts to a bijection (in fact, the identity
map\footnote{This actually implies the stronger result that $\fix(\pi)=\pix(\pi)$
for all $\pi\in\mathfrak{S}_{n}(213,231,312)$.}) on $\mathfrak{S}_{n}(213,231,312)$, proving the desired equidistribution
for this avoidance class as well.
\end{proof}
The next lemma is a straightforward consequence of Lemma \ref{l-213-231-restr}
(a)\textendash (b) and the fact that reverse-complementation preserves
$\des$ (and thus $\asc$).
\begin{lem}
\label{l-132-312-restr}Let $n\geq1$, and suppose that $\pi\in\mathfrak{S}_{n}(132,312)$.
Then\textup{:}
\begin{enumerate}
\item [\normalfont{(a)}] $\pi$ avoids $123$ if and only if $\asc(\pi)\leq1$,
and
\item [\normalfont{(b)}] $\pi$ avoids $321$ if and only if $\des(\pi)\leq1$.
\end{enumerate}
\end{lem}

Lemma \ref{l-132-312-restr} then implies that $\Phi_{132,312}$ restricts
to a bijection on both $\mathfrak{S}_{n}(123,132,312)$ and $\mathfrak{S}_{n}(132,312,321)$.
Hence, we have verified Theorem \ref{t-main} for $\Pi=\{123,132,312\}$
and $\Pi=\{132,312,321\}$.

In analogy to Lemma \ref{l-213-231-restr} (c), $\pi\in\mathfrak{S}_{n}(132,312)$
avoids $231$ if and only if $\ad(\pi)=D^{n-k-1}A^{k}$ where $k=\asc(\pi)$.
But, for such a $\pi$, we have
\[
(\Psi\circ\ad\circ\,rc)(\pi)=(\Psi\circ r\circ\ad)(\pi)=\Psi(A^{k}D^{n-k-1})=A^{k}D^{n-k-1},
\]
which is not in the appropriate form for the corresponding permutation
to avoid $231$. Therefore, this does not lead to an analogous equidistribution
over $\mathfrak{S}_{n}(132,231,312)$, and in fact the statistics
$(\fix,\des)$ and $(\pix,\des)$ have different distributions over
this avoidance class.

Lastly, we know from the Erd\H{o}s\textendash Szekeres theorem \cite{Erdoes1935}
that $\left|\mathfrak{S}_{n}(123,312,321)\right|=0$ for all $n\geq5$.
Then the map defined by
\[
1\mapsto1,\quad12\mapsto12,\quad21\mapsto21,\quad132\mapsto132,\quad213\mapsto231,\quad231\mapsto213,\quad2143\mapsto2143
\]
is a bijection on $\mathfrak{S}_{n}(123,312,321)$ which sends the
number of fixed points to the number of pixed points while preserving
the number of descents. This completes the proof of Theorem~\ref{t-main}.

\section{\label{s-132}From \texorpdfstring{$132$}{132}-avoiding derangements
to \texorpdfstring{$132$}{132}-avoiding desarrangements}

Our last aim is to give a bijective proof for the fact that $\left|\mathcal{D}_{n}(132)\right|=\left|\mathfrak{D}_{n}(132)\right|$
for all $n\geq0$, in order to complete a bijective proof of Theorem
\ref{t-derdes}. We will accomplish this task by composing three bijections,
all of which involve Dyck paths. Let us start by introducing two of
these bijections, one due to Krattenthaler \cite{Krattenthaler2001}
and the other to Elizalde and Deutsch \cite{Elizalde2003}, before
constructing the third.

\subsection{Dyck paths and two known bijections}

Recall that a \textit{Dyck path} of semilength $n$ is a lattice path
in $\mathbb{Z}^{2}$ between $(0,0)$ and $(2n,0)$ consisting of
\textit{up steps} $(1,1)$ and\textit{ down steps} $(1,-1)$ which
never traverses below the $x$-axis. Let $\mathcal{P}_{n}$ denote
the set of Dyck paths of semilength $n$. Dyck paths are often represented
as words over the alphabet $\{U,D\}$ where $U$ refers to an up step
and $D$ a down step.

Following Elizalde and Pak \cite{Elizalde2004}, a \textit{tunnel}
of $\mu\in\mathcal{P}_{n}$ is a horizontal segment between two lattice
points of $\mu$ which intersects $\mu$ only in those two points
and which stays below $\mu$. Additionally, a tunnel of $\mu\in\mathcal{P}_{n}$
is called \textit{centered} if the $x$-coordinate of its midpoint
is $n$; in other words, it is centered with respect to the vertical
line through the middle of the path $\mu$. A \textit{hill} is an
up step whose initial lattice point is on the $x$-axis, followed
by a down step. See Figure \ref{f-dyck} for an example of a Dyck
path with 1 centered tunnel and 3 hills.

\begin{figure}
\begin{center}
\begin{tikzpicture}[scale=0.75]
\draw[pathcolorlight] (5,1) -- (11,1);
\drawlinedots{4,5,6,7,8,9,10,11,12,13,14}{0,1,2,3,2,3,2,1,2,1,0};
\drawlinedotsred{0,1,2,3,4}{0,1,0,1,0};
\drawlinedotsred{14,15,16}{0,1,0};
\end{tikzpicture}
\end{center}
\vspace{-10bp}\caption{\label{f-dyck}The Dyck path $UDUDUUUDUDDUDDUD$, which has one centered
tunnel and three hills. The centered tunnel is illustrated with a
dotted blue line, while the hills are highlighted in red.}
\end{figure}
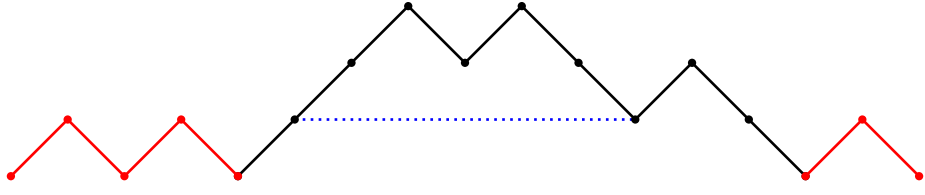

Both $\mathfrak{S}_{n}(132)$ and $\mathcal{P}_{n}$ are counted by
the Catalan numbers, and there are numerous known bijections between
them. One such bijection, which we shall denote $\kappa$, is due
to Krattenthaler~\cite{Krattenthaler2001}. While Krattenthaler's
definition of $\kappa$ is iterative, we will present an equivalent
recursive description of this bijection (see \cite[Section 3.1]{Claesson2008/09}).
First, $\kappa$ sends the empty permutation to the empty path. Now,
notice that any $\pi\in\mathfrak{S}_{n}(132)$ where $n\geq1$ can
be uniquely decomposed as $\pi=\tau n\rho$, where $\tau$ and $\rho$
are $132$-avoiding permutations and where every letter of $\tau$
is larger than every letter of $\rho$. Then $\kappa(\pi)$ is defined
to be the concatenation
\[
\kappa(\pi)\coloneqq U\kappa(\std(\tau))D\kappa(\rho),
\]
where we recall that $\std(\tau)$ is the standardization of $\tau$.
\begin{example}
\label{eg-krattenthaler}Take $\pi=987645213\in\mathfrak{S}_{9}(132)$.
Then {\allowdisplaybreaks
\begin{align*}
\kappa(\pi) & =UD\kappa(87645213)\\
 & =UDUD\kappa(7645213)\\
 & =UDUDUD\kappa(645213)\\
 & =UDUDUDUD\kappa(45213)\\
 & =UDUDUDUDU\kappa(1)D\kappa(213)\\
 & =UDUDUDUDUUDDU\kappa(21)D\\
 & =UDUDUDUDUUDDUUD\kappa(1)D\\
 & =UDUDUDUDUUDDUUDUDD.
\end{align*}
}See the top part of Figure \ref{f-132dertodes} for a visualization
of this example.
\end{example}

Observe that, in the above example, $\pi$ is a derangement and $\kappa(\pi)$
has no central tunnels. Indeed, it was proven by Elizalde and Pak
\cite[Proposition 4]{Elizalde2004} that, for all $\pi\in\mathfrak{S}_{n}(132)$,
the number of fixed points of $\pi$ is equal to the number of centered
tunnels of $\kappa(\pi)$.\footnote{Actually, Elizalde and Pak worked with a variant of Krattenthaler's
bijection in which the output path $\kappa(\pi)$ is reflected about
the vertical line through its middle, but this reflection does not
change the number of centered tunnels.} Hence, $\kappa$ restricts to a bijection from $\mathcal{D}_{n}(132)$
to the set $\mathcal{P}_{n}^{(t)}$ of Dyck paths of semilength $n$
with no central tunnels.

Next, we describe a bijection $\chi\colon\mathcal{P}_{n}\rightarrow\mathcal{P}_{n}$
due to Elizalde and Deutsch \cite{Elizalde2003}. Fix $\mu\in\mathcal{P}_{n}$,
and let $\sigma\in\mathfrak{S}_{2n}$ be the permutation whose letters
are
\[
\sigma_{i}=\begin{cases}
\frac{i+1}{2}, & \text{if }i\text{ is odd},\\
2n+1-\frac{i}{2}, & \text{otherwise, if }i\text{ is even.}
\end{cases}
\]
(So $\sigma_{1}=1$, $\sigma_{2}=2n$, $\sigma_{3}=2$, $\sigma_{4}=2n-1$,
and so on.) Observe that every up step of $\mu$ has a corresponding
down step with which it determines a tunnel; match every up step with
its corresponding down step. We iterate through all $i$ from $1$
to $n$; at each iteration, we ``read'' the $\sigma_{i}$th step of
$\mu$ (so that the steps of $\mu$ are read in an alternating manner)
and define the $i$th step of $\chi(\mu)$ based on whether the step
matched with the $\sigma_{i}$th step of $\mu$ has already been read.
If this matching step has not yet been read, then define the $i$th
step of $\chi(\mu)$ to be an up step; otherwise, if it has been read,
then define the $i$th step of $\chi(\mu)$ to be a down step.
\begin{example}
\label{eg-ElizaldePak}Continuing Example \ref{eg-krattenthaler},
take $\mu=\kappa(\pi)=UDUDUDUDUUDDUUDUDD$.
\begin{enumerate}
\item The first step of $\mu$ is matched with the second step of $\mu$,
which has not been read. So, the first step of $\chi(\mu)$ is a $U$.
\item The last step of $\mu$ is matched with the 13th step of $\mu$, which
has not been read. So, the second step of $\chi(\mu)$ is a $U$.
\item The second step of $\mu$ is matched with the first step of $\mu$,
which has been read. So, the third step of $\chi(\mu)$ is a $D$.
\item The penultimate step of $\mu$ is matched with the antepenultimate
step of $\mu$, which has not been read. So, the fourth step of $\chi(\mu)$
is an $U$.
\end{enumerate}
Continuing in this manner, we have $\chi(\mu)=UUDUUDDUUDDDUUDUDD$.
See the middle part of Figure \ref{f-132dertodes}.
\end{example}

In the above example, $\mu$ has no central tunnels and $\chi(\mu)$
has no hills. Elizalde and Deutsch proved that, for all $\mu\in\mathcal{P}_{n}$,
the number of central tunnels of $\mu$ is equal to the number of
hills of $\chi(\mu)$, so $\chi$ restricts to a bijection from $\mathcal{P}_{n}^{(t)}$
to the set $\mathcal{P}_{n}^{(h)}$ of Dyck paths of semilength $n$
with no hills.

\subsection{From \texorpdfstring{$132$}{132}-avoiding desarrangements to Dyck
paths with no hills}

Lastly, we define a modification of Krattenthaler's bijection $\kappa$
which will provide a bijection from $132$-avoiding desarrangements
to Dyck paths with no hills. As with $\kappa$, our modified bijection
$\hat{\kappa}$ is defined recursively and sends the empty permutation
to the empty path. Given a nonempty $\pi\in\mathfrak{D}_{n}(132)$,
we write $\pi=\tau n\rho$. Note that $\tau$ is a desarrangement
whenever $\pi$ is a desarrangement, which allows us to define $\hat{\kappa}(\pi)$
as
\[
\hat{\kappa}(\pi)\coloneqq\begin{cases}
U\kappa(\tau)D, & \text{if }\rho\text{ is empty,}\\
\hat{\kappa}(\std(\tau))U\kappa(\rho)D, & \text{otherwise, if }\rho\text{ is nonempty.}
\end{cases}
\]

\begin{prop}
For all $n\geq0$, the map $\hat{\kappa}$ is a bijection from $\mathfrak{D}_{n}(132)$
to $\mathcal{P}_{n}^{(h)}$.
\end{prop}

\begin{proof}
We proceed by induction, with the base case ($n=0$) being trivial.
Suppose that $\hat{\kappa}\colon\mathfrak{D}_{k}(132)\rightarrow\mathcal{P}_{k}^{(h)}$
is a bijection for all $k$ up to $n-1$. We first prove that $\hat{\kappa}(\pi)$
has no hills for all $\pi\in\mathfrak{D}_{n}(132)$, and then we prove
that $\hat{\kappa}$ is a bijection from $\mathfrak{D}_{n}(132)$
to $\mathcal{P}_{n}^{(h)}$.

Let $\pi\in\mathfrak{D}_{n}(132)$ and write $\pi=\tau n\rho$. If
$\rho$ is empty, then $\tau$ is nonempty (as there is no desarrangement
of length $1$), so $\kappa(\tau)$ is nonempty and thus $\hat{\kappa}(\pi)=U\kappa(\tau)D$
has no hills. Now, suppose that $\rho$ is nonempty. Then $U\kappa(\rho)D$
has no hills. Moreover, the induction hypothesis tells us that $\hat{\kappa}(\std(\tau))$
has no hills either, so $\hat{\kappa}(\pi)=\hat{\kappa}(\std(\tau))U\kappa(\rho)D$
has no hills.

It remains to prove that $\hat{\kappa}$ is a bijection. Let $\mu\in\mathcal{P}_{n}^{(h)}$.
We can uniquely decompose $\mu$ as the concatenation $\mu=\mu_{1}U\mu_{2}D$,
where $\mu_{1}$ and $\mu_{2}$ are Dyck paths of smaller semilength;
this is the \textit{last return decomposition} of Dyck paths. Suppose
that $\mu_{1}$ is nonempty. Observe that $\mu_{1}$ does not have
any hills, so by the induction hypothesis, $\hat{\kappa}^{-1}(\mu_{1})\in\mathfrak{D}_{k}(132)$
where $k$ is the semilength of $\mu_{1}$. Then, $\hat{\kappa}^{-1}(\mu)$
is the concatenation
\[
\hat{\kappa}^{-1}(\mu)=\overrightarrow{\hat{\kappa}^{-1}(\mu_{1})}n\kappa^{-1}(\mu_{2});
\]
here, $\overrightarrow{\hat{\kappa}^{-1}(\mu_{1})}$ is the unique
permutation on the letters $n-k,n-k+1,\dots,n-1$ whose standardization
is $\hat{\kappa}^{-1}(\mu_{1})$. Now, suppose that $\mu_{1}$ is
empty. If $\kappa^{-1}(\mu_{2})$ is a desarrangement, then 
\[
\hat{\kappa}^{-1}(\mu)=\kappa^{-1}(\mu_{2})n.
\]
Otherwise, if $\kappa^{-1}(\mu_{2})$ is not a desarrangement, then
\[
\hat{\kappa}^{-1}(\mu)=n\kappa^{-1}(\mu_{2}).
\]
Hence, $\hat{\kappa}$ is a bijection from $\mathfrak{D}_{n}(132)$
to $\mathcal{P}_{n}^{(h)}$.
\end{proof}
Because $\hat{\kappa}\colon\mathfrak{D}_{n}(132)\rightarrow\mathcal{P}_{n}^{(h)}$
is a bijection, the composition 
\[
\hat{\kappa}^{-1}\circ\chi\circ\kappa
\]
is a bijection from $\mathcal{D}_{n}(132)$ to $\mathfrak{D}_{n}(132)$.
Together with the bijections we used to prove Theorem \ref{t-main},
this yields a bijective proof of Theorem \ref{t-derdes}.
\begin{example}
Let $\pi=756348921\in\mathfrak{D}_{9}(132)$. Then 
\begin{align*}
\hat{\kappa}(\pi) & =\hat{\kappa}(534126)U\kappa(21)D\\
 & =U\kappa(53412)DUUD\kappa(1)D\\
 & =UUD\kappa(3412)DUUDUDD\\
 & =UUDU\kappa(1)D\kappa(12)DUUDUDD\\
 & =UUDUUDDU\kappa(1)DDUUDUDD\\
 & =UUDUUDDUUDDDUUDUDD.
\end{align*}
This path appeared previously in Example \ref{eg-ElizaldePak}. Combining
the above with Examples~\ref{eg-krattenthaler}\textendash \ref{eg-ElizaldePak},
we obtain
\[
(\hat{\kappa}^{-1}\circ\chi\circ\kappa)(987645213)=756348921;
\]
see Figure \ref{f-132dertodes}.
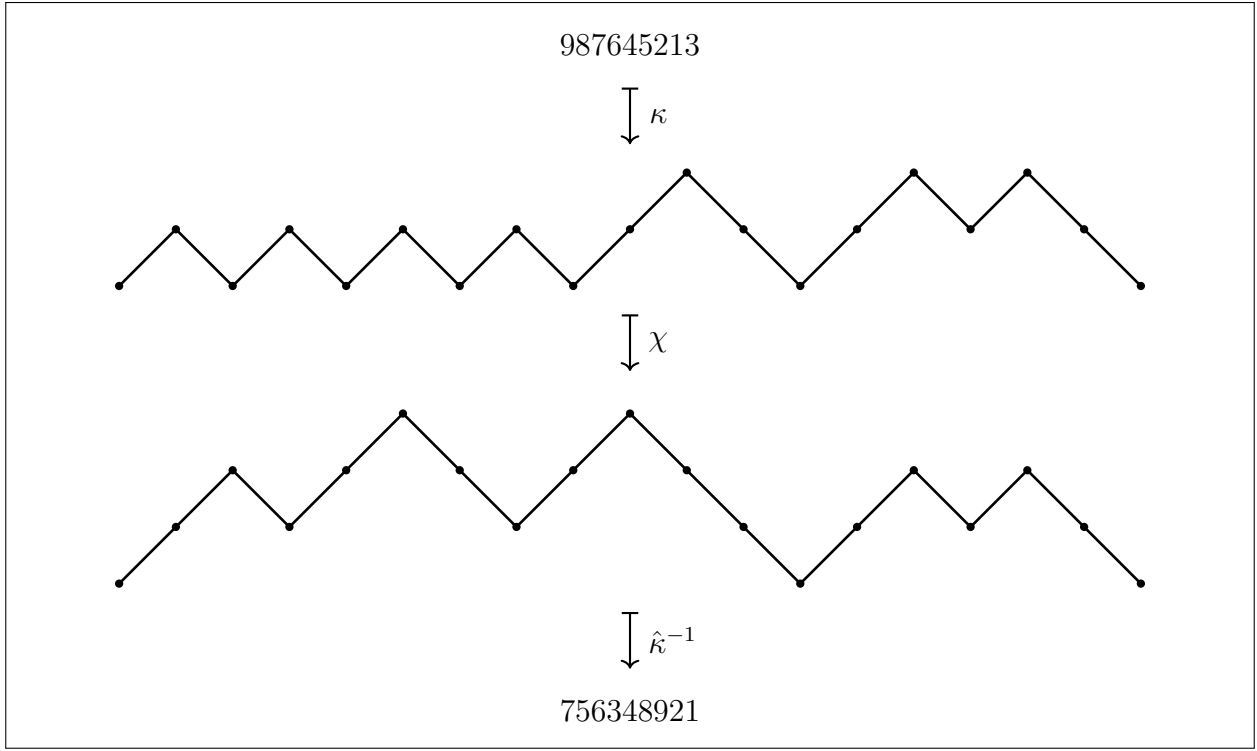
\begin{figure}
\noindent\fbox{\begin{minipage}[t]{1\columnwidth - 2\fboxsep - 2\fboxrule}%
\vspace{5bp}
\begin{center}
\begin{tikzpicture}[scale=0.75]

\node at (9,4.25) {$987645213$};

\draw[|->, thick] (9,3.5) -- (9,2.5);
\node at (9.5,3) {$\kappa$};

\drawlinedots{0,1,2,3,4,5,6,7,8,9,10,11,12,13,14,15,16,17,18}{0,1,0,1,0,1,0,1,0,1,2,1,0,1,2,1,2,1,0};

\draw[|->, thick] (9,-0.5) -- (9,-1.5);
\node at (9.5,-1) {$\chi$};

\drawlinedots{0,1,2,3,4,5,6,7,8,9,10,11,12,13,14,15,16,17,18}{-5.25,-4.25,-3.25,-4.25,-3.25,-2.25,-3.25,-4.25,-3.25,-2.25,-3.25,-4.25,-5.25,-4.25,-3.25,-4.25,-3.25,-4.25,-5.25};

\draw[|->, thick] (9,-5.75) -- (9,-6.75);
\node at (9.75,-6.25) {$\hat{\kappa}^{-1}$};

\node at (9,-7.5) {$756348921$};

\end{tikzpicture}
\end{center}
\vspace{-10bp}%
\end{minipage}}\caption{\label{f-132dertodes} An example illustrating the bijection $\hat{\kappa}^{-1}\circ\chi\circ\kappa\colon\mathcal{D}_{n}(132)\rightarrow\mathfrak{D}_{n}(132)$.}
\end{figure}
\end{example}

\section{\label{s-open}Open problems}

To conclude this paper, we present two open problems for future investigation.

Theorem \ref{t-derdes} concerns sets $\Pi$ of length 3 patterns
such that $\left|\mathcal{D}_{n}(\Pi)\right|=\left|\mathfrak{D}_{n}(\Pi)\right|$
for all $n\geq0$, but more generally, one can ask for conditions
under which this property holds for sets $\Pi$ containing longer
patterns.
\begin{problem}
\label{pb-trokya}Let $\Pi$ be an arbitrary set of patterns. Find
necessary conditions and sufficient conditions on $\Pi$ for $\left|\mathcal{D}_{n}(\Pi)\right|=\left|\mathfrak{D}_{n}(\Pi)\right|$
for all $n\geq0$.
\end{problem}

One can also pose the analogous problem for $\fix$ and $\pix$ being
equidistributed.
\begin{problem}
Let $\Pi$ be an arbitrary set of patterns. Find necessary conditions
and sufficient conditions on $\Pi$ for the statistics $\fix$ and
$\pix$ being equidistributed over $\mathfrak{S}_{n}(\Pi)$ for all
$n\geq0$.
\end{problem}

\vspace{10bp}

\noindent \textbf{Acknowledgments.} We thank Justin Troyka for suggesting
Problem \ref{pb-trokya} at the 24th International Conference on Permutation
Patterns. The authors were partially supported by NSF grant DMS-2316181.

\bibliographystyle{plain}
\addcontentsline{toc}{section}{\refname}\bibliography{bibliography}

\end{document}